\documentclass[11pt]{amsart}

\usepackage[T1]{fontenc}
\usepackage{lmodern}
\usepackage{amsmath,amssymb,amsthm,mathtools}
\usepackage{booktabs}
\usepackage[margin=1in]{geometry}
\usepackage{microtype}
\usepackage[hidelinks]{hyperref}

\newtheorem{theorem}{Theorem}[section]
\newtheorem{proposition}[theorem]{Proposition}

\theoremstyle{remark}

\newcommand{\F}{\mathbb{F}}

\title{A Two-Dimensional Counterexample to Radical Equality in Primitive Axial Algebras}
\author[B. Peng]{
  Bo Peng \\ \vspace{2pt}
  \textmd{\small Institute of Mathematical Sciences, ShanghaiTech University}
}
\date{\today}

\hypersetup{
  pdftitle={A Two-Dimensional Counterexample to Radical Equality in Primitive Axial Algebras},
  pdfauthor={Bo Peng}
}

\begin{document}
\begin{abstract}
  Let $R(A,X)$ denote the largest ideal of a primitive axial algebra $(A,X)$
  that contains no axis from the specified generating set $X$, and let $J(A)$
  be the intersection of the maximal ideals of $A$. Mamontov, Shpectorov, and
  Zhelyabin asked whether $R(A,X)=J(A)$ always holds. We give a negative answer.
  Over every field of characteristic different from $2$, the two-dimensional
  commutative algebra with basis $a,b$ and multiplication
  \[
    a^2=a,\qquad ab=2b,\qquad b^2=b
  \]
  is a primitive axial algebra for an explicit fusion law and the generating
  set $X=\{a,b\}$. Its complete ideal lattice is
  $0<\F b<A$, whence $R(A,X)=0$ and $J(A)=\F b$; in particular, the primitive
  axis $b$ lies in $J(A)$. Over $\mathbb{C}$, this axial presentation is
  equivalent to the previously classified
  $D(-1),\{e_2,a_6\}$ presentation with fusion law $F_{D3}$. Thus the algebra,
  the axial structure, and the vanishing of the axial radical are known; the
  point is the resulting Jacobson-radical computation and its consequence for
  the radical-equality question.
\end{abstract}

\maketitle


\section{Introduction}

Axial algebras are commutative, generally nonassociative algebras generated by
semisimple idempotents whose eigenspaces obey a prescribed fusion law. The
structure theory of primitive axial algebras associates to a specified
generating set of axes $X$ a distinguished ideal: the axial radical
$R(A,X)$, the unique largest ideal containing no member of $X$; see
\cite{KhasrawMcInroyShpectorov2020}. The dependence on $X$ is essential, so we
retain it in the notation throughout this note.

Mamontov, Shpectorov, and Zhelyabin define the Jacobson radical of an axial
algebra to be
\[
  J(A)=\bigcap\{M\mid M\text{ is a maximal ideal of }A\},
\]
with $J(A)=A$ if no maximal ideals exist \cite{MamontovShpectorovZhelyabin2026}.
They prove, in particular, that
\[
  R(A,X)\subseteq J(A)
\]
for every primitive axial algebra. If a Frobenius form is fixed, their full
chain is $R(A,X)\subseteq J(A)\subseteq A^\perp$. They then ask whether the
first inclusion is always an equality, equivalently whether $J(A)$ can contain
a primitive generating axis
\cite[Question~9.1]{MamontovShpectorovZhelyabin2026}. The same question is
recorded as Problem~3.6 in \cite{GorshkovShpectorov2026}.

We show that the inclusion can be strict in dimension two. The example is
defined over every field of characteristic different from $2$, and its ideal
lattice can be determined without any classification theorem. We then identify
the corresponding complex axial presentation inside the two-dimensional
classification of Kaygorodov, Mart\'{\i}n Gonz\'{a}lez, and
P\'{a}ez-Guill\'{a}n \cite{KaygorodovMartinPaez2025}. This identification is
included both to locate the example accurately in the literature and to make
clear that no new family of axial algebras is being claimed.

\section{The axial structure}

Let $\F$ be a field with $\operatorname{char}\F\ne 2$. Consider the fusion law
on $\{0,1,2\}\subseteq\F$ given by Table~\ref{tab:fusion}. A blank entry means
the empty set.

\begin{table}[ht]
  \centering
  \renewcommand{\arraystretch}{1.15}
  \begin{tabular}{c|ccc}
    $\star$ & $0$           & $1$           & $2$           \\
    \hline
    $0$     & $\{0,1\}$     & $\varnothing$ & $\varnothing$ \\
    $1$     & $\varnothing$ & $\{1\}$       & $\{2\}$       \\
    $2$     & $\varnothing$ & $\{2\}$       & $\{2\}$
  \end{tabular}
  \caption{The fusion law used in the counterexample.}
  \label{tab:fusion}
\end{table}

For an idempotent $x$, write $A_\lambda(x)$ for the $\lambda$-eigenspace of
left multiplication $L_x$. As usual, an eigenspace is understood to be zero
when the corresponding scalar is not in the spectrum.

\begin{theorem}\label{thm:counterexample}
  Let $A=\F a\oplus\F b$ be the commutative algebra defined by
  \begin{equation}\label{eq:multiplication}
    a^2=a,\qquad ab=2b,\qquad b^2=b.
  \end{equation}
  With the fusion law in Table~\ref{tab:fusion} and the specified generating set
  $X=\{a,b\}$, the pair $(A,X)$ is a primitive axial algebra. Moreover,
  \[
    R(A,X)=0
    \qquad\text{and}\qquad
    J(A)=\F b.
  \]
  Consequently, $R(A,X)\subsetneq J(A)$ and the primitive axis $b$ belongs to
  $J(A)$.
\end{theorem}

\begin{proof}
  Both $a$ and $b$ are idempotents, and $X$ is a basis of $A$, hence generates
  $A$. In the ordered basis $(a,b)$, their adjoint operators are
  \[
    L_a=
    \begin{pmatrix}
      1 & 0 \\
      0 & 2
    \end{pmatrix},
    \qquad
    L_b=
    \begin{pmatrix}
      0 & 0 \\
      2 & 1
    \end{pmatrix}.
  \]
  It follows that
  \begin{align*}
    A_1(a) & =\F a, & A_2(a) & =\F b,     & A_0(a) & =0, \\
    A_1(b) & =\F b, & A_0(b) & =\F(a-2b), & A_2(b) & =0.
  \end{align*}
  Thus both operators are diagonalizable and both $1$-eigenspaces are
  one-dimensional, so $a$ and $b$ are primitive.

  It remains to check the fusion law. For the axis $a$, the only nonvacuous
  products of eigenspaces are represented by
  \[
    a^2=a\in A_1(a),\qquad
    ab=2b\in A_2(a),\qquad
    b^2=b\in A_2(a).
  \]
  For the axis $b$, put $c=a-2b$. Then
  \[
    bc=0,
    \qquad
    c^2=a-4b=c-2b\in A_0(b)+A_1(b),
    \qquad
    b^2=b\in A_1(b).
  \]
  These are exactly the nonvacuous requirements in Table~\ref{tab:fusion}.
  Hence $(A,X)$ is a primitive axial algebra for this fusion law.

  We next determine every ideal of $A$. Any nonzero proper ideal is
  one-dimensional, say $I=\F v$ with $v=xa+yb\ne0$. Since $A$ is generated as a
  vector space by $a$ and $b$, the subspace $I$ is an ideal precisely when
  $L_a(v)$ and $L_b(v)$ both belong to $\F v$. In coordinates,
  \[
    L_a(v)=xa+2yb,
    \qquad
    L_b(v)=(2x+y)b.
  \]
  The two collinearity conditions are therefore
  \begin{equation}\label{eq:ideal-equations}
    \det\begin{pmatrix}x&x\\y&2y\end{pmatrix}=xy=0,
    \qquad
    \det\begin{pmatrix}x&0\\y&2x+y\end{pmatrix}=x(2x+y)=0.
  \end{equation}
  If $x=0$, then $I=\F b$, which is indeed an ideal. If $x\ne0$, the first
  equation gives $y=0$, while the second gives $2x^2=0$, contradicting
  $\operatorname{char}\F\ne2$. Thus the complete ideal lattice is
  \[
    0<\F b<A.
  \]
  In particular, $\F b$ is the unique maximal ideal, and so $J(A)=\F b$.

  Finally, the ideals $\F b$ and $A$ both contain an axis from the specified set
  $X$: the former contains $b$, and the latter contains both $a$ and $b$.
  Therefore zero is the unique ideal avoiding every axis in $X$, and
  $R(A,X)=0$. This proves all the assertions.
\end{proof}

Taking $\F=\mathbb{Q}$ already gives a counterexample to the unrestricted
question. The field-general form of Theorem~\ref{thm:counterexample} is included
only because the same proof requires no additional work.

\section{Relation with the known two-dimensional classification}

Kaygorodov, Mart\'{\i}n Gonz\'{a}lez, and P\'{a}ez-Guill\'{a}n list the complex
two-dimensional axial algebras $D(\beta)$ with multiplication
\begin{equation}\label{eq:D-beta}
  {e_1}^2=e_1,\qquad e_1e_2=\beta e_2,\qquad {e_2}^2=e_2
\end{equation}
and several possible specified generating sets and fusion laws
\cite[Example~1.1 and Tables~4--5]{KaygorodovMartinPaez2025}. Their notation
includes
\[
  a_6=e_1+(1-2\beta)e_2.
\]

\begin{proposition}\label{prop:classification}
  After scalar extension to $\mathbb{C}$, the axial algebra in
  Theorem~\ref{thm:counterexample} is the presentation
  \[
    D(-1),\qquad X=\{e_2,a_6\},\qquad \text{fusion law }F_{D3}.
  \]
\end{proposition}

\begin{proof}
  Set $\beta=-1$, $b=e_2$, and $a=a_6=e_1+3e_2$. From
  \eqref{eq:D-beta},
  \[
    a^2=e_1-6e_2+9e_2=e_1+3e_2=a,
    \qquad
    ab=-e_2+3e_2=2b,
    \qquad
    b^2=b.
  \]
  Thus the multiplication is \eqref{eq:multiplication}, and the specified set
  $\{e_2,a_6\}$ is exactly $\{b,a\}$. At $\beta=-1$, the $F_{D3}$ eigenvalue set
  $\{1,1-\beta,0\}$ is $\{1,2,0\}$, and its nonzero fusion products are
  \[
    1\star2=\{2\},\qquad
    2\star2=\{2\},\qquad
    0\star0=\{0,1\},
  \]
  which is Table~\ref{tab:fusion}.
\end{proof}

The use of $\beta=-1$ avoids any ambiguity caused by the orbit-representative
convention in the classification. A formal substitution $\beta=2$ in
\eqref{eq:D-beta} with the generating set $\{e_1,e_2\}$ also reproduces
\eqref{eq:multiplication} and the displayed fusion table, but the literal label
$D(2)$ depends on the authors' choice of representatives.

The same source states that, among its displayed generating sets on
$D(\beta)$, only $\{e_1,a_6\}$ has nonzero axial radical
\cite[p.~5]{KaygorodovMartinPaez2025}. In particular, the equivalent
$\{e_2,a_6\}$ presentation above was already known to have axial radical zero.
Accordingly, Proposition~\ref{prop:classification} is a prior-art
identification, not a claim to a new two-dimensional algebra, fusion law, or
axial-radical computation.

\section{Compatibility with the Frobenius-form results}

The counterexample does not contradict the positive equality result under
nonsingularity of the generating axes. Indeed, let $\langle\ ,\ \rangle$ be an
associating symmetric bilinear form on $A$. Associativity with the product gives
\[
  \langle a,b\rangle
  =\langle a^2,b\rangle
  =\langle a,ab\rangle
  =2\langle a,b\rangle,
\]
so $\langle a,b\rangle=0$. Similarly,
\[
  2\langle b,b\rangle
  =\langle ab,b\rangle
  =\langle a,b^2\rangle
  =\langle a,b\rangle=0.
\]
Since $\operatorname{char}\F\ne2$, we obtain $\langle b,b\rangle=0$, and hence
$b$ lies in the radical of every such form. Thus the primitive axis contained
in $J(A)$ is necessarily singular. This is precisely the situation excluded by
the hypothesis that all generating axes be nonsingular, under which
$R(A,X)=J(A)=A^\perp$ \cite{MamontovShpectorovZhelyabin2026}.

\section{Conclusion}

The two-dimensional algebra in Theorem~\ref{thm:counterexample} gives
\[
  R(A,X)=0\subsetneq\F b=J(A),
\]
and therefore answers the unrestricted radical-equality question in the
negative. The example also shows why the specified generating set $X$ cannot
be suppressed: the obstruction is exactly that one of its primitive axes lies
in the unique maximal ideal. Its small dimension and known place in the
$D(\beta)$ classification make the failure transparent.

\bibliographystyle{alpha}
\bibliography{references}

@misc{GorshkovShpectorov2026,
  author       = {Gorshkov, I. and Shpectorov, S.},
  title        = {Axial Algebras: Questions and Conjectures},
  year         = {2026},
  eprint       = {2606.30048},
  archiveprefix = {arXiv},
  primaryclass = {math.RA},
  url          = {https://arxiv.org/abs/2606.30048},
  note         = {arXiv:2606.30048 [math.RA]}
}

@article{KaygorodovMartinPaez2025,
  author  = {Kaygorodov, Ivan and Mart{\'{\i}}n Gonz{\'{a}}lez, C{\'{a}}ndido and P{\'{a}}ez-Guill{\'{a}}n, Pilar},
  title   = {Central Extensions of Axial Algebras},
  journal = {Journal of Algebra},
  volume  = {662},
  pages   = {797--831},
  year    = {2025},
  doi     = {10.1016/j.jalgebra.2024.09.001},
  eprint  = {2211.00334},
  archiveprefix = {arXiv},
  primaryclass = {math.RA},
  note    = {Preprint: arXiv:2211.00334 [math.RA]}
}

@article{KhasrawMcInroyShpectorov2020,
  author  = {Khasraw, S. M. S. and McInroy, J. and Shpectorov, S.},
  title   = {On the Structure of Axial Algebras},
  journal = {Transactions of the American Mathematical Society},
  volume  = {373},
  number  = {3},
  pages   = {2135--2156},
  year    = {2020}
}

@misc{MamontovShpectorovZhelyabin2026,
  author       = {Mamontov, Andrey and Shpectorov, Sergey and Zhelyabin, Victor},
  title        = {Radicals in Primitive Axial Algebras},
  year         = {2026},
  eprint       = {2602.11984},
  archiveprefix = {arXiv},
  primaryclass = {math.RA},
  url          = {https://arxiv.org/abs/2602.11984},
  note         = {arXiv:2602.11984 [math.RA]}
}

\end{document}